\documentclass[11pt,a4paper]{article}

\usepackage[T1]{fontenc}
\usepackage[utf8]{inputenc}
\usepackage[margin=2.5cm]{geometry}
\usepackage{mathtools}
\usepackage{amsthm}
\usepackage{newtxtext}
\usepackage[subscriptcorrection]{newtxmath}
\usepackage{microtype}
\usepackage{graphicx}
\usepackage{hyperref}
\usepackage[
  backend=biber,
  style=alphabetic,
  natbib=true,
  maxbibnames=40,
  doi=false,
  isbn=false,
  url=false,
  eprint=false
]{biblatex}
\theoremstyle{plain}
\newtheorem{theorem}{Theorem}
\newtheorem{lemma}{Lemma}
\newtheorem{proposition}{Proposition}
\newtheorem{corollary}{Corollary}
\theoremstyle{definition}
\newtheorem{assumption}{Assumption}
\newcommand{\R}{\mathbb{R}}
\newcommand{\E}{\mathbb{E}}
\newcommand{\Law}{\mathcal{L}}
\newcommand{\norm}[1]{\lVert#1\rVert}
\usepackage{authblk}

\title{Theoretical guarantees for stochastic gradient Langevin dynamics}
\author[1]{Daniel Paulin}
\author[2]{Peter A. Whalley}

\affil[1]{College of Computing and Data Science\\
Nanyang Technological University, Singapore}

\affil[2]{Department of Statistics\\
University of Warwick, United Kingdom}

\begin{document}
\maketitle
\begin{abstract}
We prove asymptotic bias bounds for stochastic gradient Langevin dynamics in Wasserstein distance of order two. We assume that the negative log-density is strongly convex with a Lipschitz gradient, and that the stochastic gradient estimator is unbiased with an error satisfying a mean-square Lipschitz condition. The bounds are of order $h$ under a fourth moment assumption on the stochastic gradient error and of order $h^{1/2}$ under only a second moment assumption, where $h$ is the stepsize. A spiked-noise example shows that a second moment assumption alone is insufficient for a bound of order $h$ that is uniform over noise distributions with a fixed variance.
\end{abstract}

\section{Introduction}

We consider sampling from a probability measure $\pi(dx)\propto\exp\left(-V(x)\right)dx$ on $\R^d$, an important computational task in modern statistics underpinning Bayesian inference and uncertainty quantification (see \cite{Roberts2004,liu2001monte}). A classical approach uses the overdamped Langevin dynamics, constructed to have invariant measure $\pi$:
\[
 X_0\sim\mu,\qquad dX_t=-\nabla V(X_t)\,dt+\sqrt{2}\,dB_t,
\]
where $(B_t)_{t\geq0}$ is a standard $d$-dimensional Brownian motion.

These dynamics generally cannot be simulated exactly, so one resorts to numerical discretisation. The simplest choice is the Euler discretisation, which gives the unadjusted Langevin algorithm (see, e.g., \cite{durmus2019high}). Starting from $Y_0=X_0\sim\mu$, its iterates are defined by
\begin{equation}\label{eq:ULA}
 Y_{k+1}=Y_k-h\nabla V(Y_k)+\sqrt{2h}Z_{k+1},
\end{equation}
where $h>0$ is a fixed stepsize and $\sqrt{h}Z_k\coloneq B_{kh}-B_{(k-1)h}$ for $k\geq1$. Under appropriate assumptions and for sufficiently small $h$, this chain has a unique invariant measure $\pi_h$, which generally differs from $\pi$ due to discretisation error.

Evaluating the full gradient at each iteration can be expensive, particularly when the log-density involves a sum over a large dataset. Replacing it by an unbiased stochastic gradient approximation gives stochastic gradient Langevin dynamics (see \cite{WellingTeh2011} and \cite{ScalableLearningBook} for a complete overview). Specifically, we replace $\nabla V$ in \eqref{eq:ULA} by a jointly measurable stochastic gradient estimator $g(\cdot,\omega)$, where $\omega\sim\Gamma$ (a probability measure defined on a Polish space), satisfying
\[
    \E_\omega g(x,\omega)=\nabla V(x), \qquad x\in\R^d.
\]
We write $\eta(x,\omega)=g(x,\omega)-\nabla V(x)$ for the stochastic gradient error, so that $\E_\omega\eta(x,\omega)=0$ for every $x\in\R^d$. Starting from $X_0\sim\mu$, the iterates of stochastic gradient Langevin dynamics satisfy
\begin{equation}\label{eq:SGLD}
    X_{k+1}=X_k-h\nabla V(X_k)-h\eta(X_k,\omega_{k+1})+\sqrt{2h}\widetilde Z_{k+1},\qquad k\geq0.
\end{equation}
Here the pairs $(\omega_k,\widetilde Z_k)_{k\geq1}$ are independent and identically distributed with law $\Gamma\otimes\mathcal{N}(0,I_d)$, and are independent of $X_0$.

Under appropriate assumptions and for sufficiently small $h$, \eqref{eq:SGLD} has a unique invariant law, which we denote by $\widetilde\pi_h$. Unbiasedness of the stochastic gradient does not imply that $\widetilde\pi_h$ coincides with either $\pi_h$ or $\pi$ and our aim is to quantify the asymptotic bias $W_2(\widetilde\pi_h,\pi)$, where $W_2$ denotes the Wasserstein distance of order two.

Previous work on stochastic gradient Langevin dynamics has characterised the asymptotic bias through weak-error expansions \citep{VollmerZygalakisTeh2016} and provided non-asymptotic bounds in Wasserstein distance \citep{DalalyanKaragulyan2019}. For Wasserstein bounds, a standard approach is to compare the stochastic gradient scheme with the unadjusted Langevin algorithm using synchronous coupling, which gives a bound of order $h^{1/2}$ for the stochastic gradient contribution to the bias. Bounds of order $h$ in Wasserstein distance have been obtained under other assumptions. For example, \cite{LiWang2025} prove a uniform-in-time relative entropy bound of order $h^2$, which also gives a Wasserstein bound of order $h$. Their assumptions include Lipschitz conditions on each stochastic drift and its Jacobian, bounded gradient errors and a warm start. Recently, \cite{PaulinWhalley2026} obtain Wasserstein bias bounds of order $h$ for a stochastic gradient kinetic Langevin method using Gaussian convolution inequalities.

In this article, we prove a bias bound of order $h$ for \eqref{eq:SGLD} under a fourth moment assumption on the stochastic gradient error, together with a bound of order $h^{1/2}$ that can be preferable when this moment is large or infinite. Our analysis combines the Gaussian convolution techniques of \cite{PaulinWhalley2026} with the recent bias estimates for the unadjusted Langevin algorithm of \cite{PedrottiWhalley2026}. The article is organised as follows. We first state our assumptions and establish convergence of stochastic gradient Langevin dynamics to its invariant measure. We then present the asymptotic bias estimates and the resulting convergence guarantees. Finally, we use a spiked-noise example to show that a second moment assumption alone is insufficient for a bias bound of order $h$ that is uniform over noise distributions with a fixed variance.

\section{Assumptions and convergence}

For probability measures $\mu,\nu\in\mathcal P_2(\R^d)$, where $\mathcal P_2(\R^d)$ denotes the measures with finite second moments, we use the Wasserstein distance of order two
\[
 W_2(\mu,\nu)=\inf_{(U,U')\sim(\mu,\nu)} \left(\E\norm{U-U'}^2\right)^{1/2}.
\]
The infimum runs over all couplings with the prescribed marginal laws, and $\norm{\cdot}$ is the Euclidean norm.

\begin{assumption}\label{ass:potential}
The negative log-density $V\in C^2(\R^d)$ satisfies
\[
 \alpha I_d\preceq\nabla^2 V\preceq\beta I_d,\qquad 0<\alpha\leq\beta.
\]
We write $\kappa=\beta/\alpha$ for the condition number.
\end{assumption}

\begin{assumption}\label{ass:sg}
The stochastic gradient error $(x,\omega)\mapsto\eta(x,\omega)$ is jointly measurable, with $x\in\R^d$ and $\omega\sim\Gamma$ taking values in a Polish space. For all $x,y\in\R^d$,
\begin{equation}\label{eq:ass_sg}
 \E_\omega\eta(x,\omega)=0,\qquad \E_\omega\norm{\eta(x,\omega)-\eta(y,\omega)}^2\leq C_G\norm{x-y}^2,
\end{equation}
for some $C_G\geq0$, and
\begin{equation}\label{eq:moment}
\sigma_2^2=\int\E_\omega\norm{\eta(x,\omega)}^2 \pi(dx)<\infty.
\end{equation}
We also write
\[
\sigma_4^4=\int\E_\omega\norm{\eta(x,\omega)}^4 \pi(dx) \in[0,\infty].
\]
\end{assumption}

The condition in \eqref{eq:ass_sg} is sufficient for our proof and allows the variance of the stochastic gradient to depend on the position. This condition is required for us to show that \eqref{eq:SGLD} converges towards its invariant measure, $\widetilde\pi_h$. 

We assume that the stepsize satisfies
\begin{equation}\label{eq:stepsize}
 0<h\leq\min\left\{\frac1\beta,\frac{\alpha}{\alpha^2+2C_G}\right\}.
\end{equation}
Under Assumption~\ref{ass:potential} the unadjusted Langevin algorithm has a unique invariant measure $\pi_h$ in $\mathcal P_2(\R^d)$ and \cite[Theorem~1]{PedrottiWhalley2026} show that
\begin{equation}\label{eq:ULA_bias}
 \sqrt{\alpha}W_2(\pi_h,\pi)\leq 6h\beta\sqrt{d}.
\end{equation}
We first prove convergence of stochastic gradient Langevin dynamics to its invariant measure using synchronous coupling, as in \cite{LeimkuhlerPaulinWhalley2024}.

\begin{proposition}\label{prop:contraction}
Suppose Assumptions~\ref{ass:potential} and~\ref{ass:sg} and \eqref{eq:stepsize} hold. The chain \eqref{eq:SGLD} has a unique invariant law $\widetilde\pi_h$ in $\mathcal P_2(\R^d)$. If $X_0\sim\mu\in\mathcal P_2(\R^d)$, then for every integer $k\geq1$,
\begin{align}
W_2\left(\Law(X_k),\widetilde\pi_h\right)
 &\leq\{(1-\alpha h)^2+h^2C_G\}^{k/2} W_2(\mu,\widetilde\pi_h)\leq\exp(-\alpha hk/2)W_2(\mu,\widetilde\pi_h),
 \label{eq:biased_convergence}
\end{align}
where $\Law(X_k)$ denotes the law of $X_k$.
\end{proposition}

\begin{proof}
The eigenvalues of the average Hessian along the segment from $x$ to $y$ lie in $[\alpha,\beta]$, so
\begin{equation}\label{eq:drift_contraction}
 \norm{x-y-h\left(\nabla V(x)-\nabla V(y)\right)}
 \leq(1-\alpha h)\norm{x-y}.
\end{equation}
We couple two steps of \eqref{eq:SGLD}, starting from $x$ and $y$, using the same $\omega$ and Gaussian increment. Since the stochastic gradient is unbiased, the cross term has expectation zero and we have
\begin{align*}
 &\E_\omega\norm{x-y-h\left(\nabla V(x)-\nabla V(y)\right) -h\left(\eta(x,\omega)-\eta(y,\omega)\right)}^2\\
 &\hspace{6mm}=\norm{x-y-h\left(\nabla V(x)-\nabla V(y)\right)}^2
                  +h^2\E_\omega\norm{\eta(x,\omega)-\eta(y,\omega)}^2\\
 &\hspace{12mm}\leq\{(1-\alpha h)^2+h^2C_G\}\norm{x-y}^2.
\end{align*}
The chain has finite second moments whenever its initial distribution belongs to $\mathcal P_2(\R^d)$, due to \eqref{eq:ass_sg} and \eqref{eq:moment}. Further, the condition \eqref{eq:stepsize} ensures that
$
 (1-\alpha h)^2+h^2C_G\leq 1-\alpha h<1
$
and we have contraction in the complete metric space $(\mathcal P_2(\R^d),W_2)$ \citep[see][Theorem~6.18]{Villani2009}. Then, due to Banach’s fixed-point theorem, we have a unique invariant measure and iterating the inequality and using that $1-\alpha h\leq\exp(-\alpha h)$ give \eqref{eq:biased_convergence}.
\end{proof}

Proposition~\ref{prop:contraction} gives convergence to the invariant measure of the scheme. To obtain convergence bounds with respect to $\pi$, we also require a bound on $W_2(\widetilde\pi_h,\pi)$, which is the focus of the subsequent section.

\section{Asymptotic bias bounds}

\begin{theorem}\label{thm:bias}
Suppose Assumptions~\ref{ass:potential} and~\ref{ass:sg} and \eqref{eq:stepsize} hold; then
\begin{equation}\label{eq:bias}
 \sqrt{\alpha}W_2(\widetilde\pi_h,\pi) \leq 12h\beta\sqrt{d} +\min\left(\frac{h\sigma_4^2}{2},\sqrt{2h}\sigma_2\right).
\end{equation}
\end{theorem}

For fixed $\alpha$, $\beta$, $C_G$ and finite $\sigma_4$, the bound in \eqref{eq:bias} is of order $h$, but with only $\sigma_2<\infty$, it is of order $h^{1/2}$. 

\begin{proof}
Let $X_0=Y_0\sim\pi_h$, so that $Y_k\sim\pi_h$ for all $k$.
Our approach is to couple \eqref{eq:SGLD} and \eqref{eq:ULA} and bound their mean square difference uniformly in $k$. We now carefully construct the local coupling used to bound the difference between $X_k$ and $Y_k$.

For independent $Z\sim \mathcal{N}(0,I_d)$ and $\omega\sim\Gamma$, define
\[
 \rho_{h,x}=\Law\left( \sqrt{2h}Z-h\eta(x,\omega)\right),\qquad
 \gamma_h=\mathcal{N}(0,2hI_d),\qquad f_h(x)=W_2(\rho_{h,x},\gamma_h).
\]
Here $\rho_{h,x}$ and $\gamma_h$ are the distributions of the noise
added at each step of \eqref{eq:SGLD} and \eqref{eq:ULA}, respectively.
Using the same Gaussian vector in both noise terms gives
\[
 f_h(x)^2\leq h^2\E_\omega\norm{\eta(x,\omega)}^2.
\]
When $\sigma_4<\infty$, the Gaussian convolution inequality of
\cite[Theorem~G.3]{PaulinWhalley2026}, rescaled by $ \sqrt{2h}$, gives
\[
 f_h(x)^2\leq\frac{h^3}{8}\E_\omega\norm{\eta(x,\omega)}^4\quad\hbox{for $\pi$-almost every $x$}.
\]
Integrating under $\pi$ yields
\begin{equation}\label{eq:cost_pi}
 \norm{f_h}_{L^2(\pi)}\leq\min\left(\frac{h^{3/2}\sigma_4^2}{\sqrt{8}},h\sigma_2\right).
\end{equation}

Since the chains are initialised at $\pi_h$, we require a bound on
$\norm{f_h}_{L^2(\pi_h)}$. Using the same $Z$ and $\omega$ at $x$ and
$y$, we have
\[
 W_2^2(\rho_{h,x},\rho_{h,y})
 \leq h^2\E_\omega\norm{\eta(x,\omega)-\eta(y,\omega)}^2
 \leq h^2C_G\norm{x-y}^2.
\]
The reverse triangle inequality for $W_2$ therefore implies
\begin{equation}\label{eq:lipschitz}
 |f_h(x)-f_h(y)|\leq W_2(\rho_{h,x},\rho_{h,y})
                         \leq hC_G^{1/2}\norm{x-y}.
\end{equation}
Now take an optimal coupling $U\sim\pi_h$, $U'\sim\pi$ and applying the $L^2$ triangle inequality, \eqref{eq:lipschitz} and \eqref{eq:cost_pi} gives
\begin{align}
 \norm{f_h}_{L^2(\pi_h)} &\leq\norm{f_h(U')}_{L^2} +\norm{f_h(U)-f_h(U')}_{L^2}\\
&\leq\norm{f_h}_{L^2(\pi)}+hC_G^{1/2}W_2(\pi_h,\pi)\\
&\leq\min\left(\frac{h^{3/2}\sigma_4^2}{\sqrt{8}},h\sigma_2\right)+hC_G^{1/2}W_2(\pi_h,\pi).
 \label{eq:L2_pi_hat}
\end{align}

We now construct the coupling at $X_k$. Let $\mathcal F_k$ contain all randomness up to iteration $k$; then, conditional on $\mathcal F_k$, draw independent $\omega_{k+1}\sim\Gamma$ and $\widetilde Z_{k+1}\sim \mathcal{N}(0,I_d)$, and then choose $Z_{k+1}$ so that
$
 \left(\sqrt{2h}\widetilde Z_{k+1}-h\eta(X_k,\omega_{k+1}),\sqrt{2h}Z_{k+1}\right)
$
is an optimal coupling of $\rho_{h,X_k}$ and $\gamma_h$, conditionally on $\mathcal F_k$. The optimal coupling can be chosen measurably in $x$; see \citep[Corollary~5.22]{Villani2009}.

Both added noise terms have conditional expectation zero, and hence
\[
 \E\big[  \sqrt{2h}(\widetilde Z_{k+1}-Z_{k+1})-h\eta(X_k,\omega_{k+1})\mid\mathcal F_k\big]=0.
\]
Consequently, using this identity, considering the constructed coupling at each iteration and expanding the squared distance between $X_{k+1}$ and $Y_{k+1}$, we obtain
\begin{align}
 \E\left[\norm{X_{k+1}-Y_{k+1}}^2\mid\mathcal F_k\right]
 &=\norm{X_k-Y_k-h\left(\nabla V(X_k)-\nabla V(Y_k)\right)}^2 +f_h(X_k)^2\\
 &\leq(1-\alpha h)^2\norm{X_k-Y_k}^2+f_h(X_k)^2.
 \label{eq:conditional}
\end{align}

Now, by \eqref{eq:lipschitz} we have
\[
 f_h(X_k)^2\leq2f_h(Y_k)^2+2h^2C_G\norm{X_k-Y_k}^2,
\]
and substituting in \eqref{eq:conditional} and taking an unconditional expectation, using $Y_k\sim\pi_h$, gives
\begin{align*}
 \E\norm{X_{k+1}-Y_{k+1}}^2
 &\leq\{(1-\alpha h)^2+2h^2C_G\}\E\norm{X_k-Y_k}^2 +2\norm{f_h}_{L^2(\pi_h)}^2\\
 &\leq(1-\alpha h)\E\norm{X_k-Y_k}^2+2\norm{f_h}_{L^2(\pi_h)}^2,
\end{align*}
where the last inequality uses \eqref{eq:stepsize}. Since $X_0=Y_0$, summing the geometric series yields, for $k\geq1$,
\begin{equation}\label{eq:uniform}
 W_2^2\left(\Law(X_k),\pi_h\right)\leq\E\norm{X_k-Y_k}^2 \leq\frac{2\norm{f_h}_{L^2(\pi_h)}^2}{\alpha h}\{1-(1-\alpha h)^k\}.
\end{equation}
Proposition~\ref{prop:contraction} implies convergence of $\Law(X_k)$ to $\widetilde\pi_h$ in $W_2$. Taking square roots, passing to the limit in \eqref{eq:uniform}, and applying \eqref{eq:L2_pi_hat} proves
\[
 W_2(\widetilde\pi_h,\pi_h) \leq\min\left(\frac{h\sigma_4^2}{2\sqrt{\alpha}},\sqrt{\frac{2h}{\alpha}}\sigma_2\right)+\sqrt{\frac{2hC_G}{\alpha}}W_2(\pi_h,\pi).
\]
Finally, \eqref{eq:stepsize} gives $(2hC_G/\alpha)^{1/2}\leq1$, so \eqref{eq:ULA_bias} and the triangle inequality give \eqref{eq:bias} as required.
\end{proof}

\section{Iteration complexity}

Combining Theorem~\ref{thm:bias} and Proposition~\ref{prop:contraction} with the triangle inequality, we have for $X_0\sim\mu\in\mathcal P_2(\R^d)$,
\begin{align}
\sqrt{\alpha}W_2\left(\Law(X_k),\pi\right)&\leq \sqrt{\alpha}e^{-\alpha hk/2}W_2(\mu,\widetilde\pi_h)+12h\beta\sqrt{d}+\min\left\{\frac{h\sigma_4^2}{2},\sqrt{2h}\sigma_2\right\}.
 \label{eq:finite_time}
\end{align}
The first term decreases geometrically with $k$, and the remaining terms are of order $h$ when $\sigma_4<\infty$, and of order $h^{1/2}$ under only a second moment assumption. 

As is typical in convergence bounds in log-concave sampling \citep{Chewi26Book}, we consider accuracy in $\sqrt{\alpha}W_2$ as a scale-invariant quantity. Then using \eqref{eq:finite_time}, and a routine computation we can establish the following corollary to Theorem \ref{thm:bias}.

\begin{corollary}\label{cor:complexity}
Suppose Assumptions~\ref{ass:potential} and~\ref{ass:sg} hold, and let
$X_0\sim\mu\in\mathcal P_2(\R^d)$. For $\varepsilon>0$, choose
\begin{equation}\label{eq:h_epsilon}
 h=\min\left(\frac1\beta,\frac{\alpha}{\alpha^2+2C_G}, \max\left(\frac{\varepsilon}{24\beta\sqrt{d}+\sigma_4^2},\frac{\varepsilon}{48\beta\sqrt{d}+32\sigma_2^2/\varepsilon}\right)\right).
\end{equation}
Then $\sqrt{\alpha}W_2\left(\Law(X_k),\pi\right)\leq\varepsilon$ whenever
\begin{equation}\label{eq:k_epsilon}
 k\geq\left\lceil\frac{2}{\alpha h}
       \log\left(2+\frac{2\sqrt{\alpha}W_2(\mu,\pi)}{\varepsilon}\right)
       \right\rceil.
\end{equation}
In particular, for $0<\varepsilon\leq1$, it suffices to take
\begin{equation}\label{eq:complexity}
 k=\mathcal{O}\left[\left(\frac{C_G}{\alpha^2}+\frac{\kappa\sqrt{d}}{\varepsilon}+\min\left(\frac{\sigma_4^2}{\alpha\varepsilon},\frac{\sigma_2^2}{\alpha\varepsilon^2}\right)\right)\log\left(2+\frac{\sqrt{\alpha}W_2(\mu,\pi)}{\varepsilon}\right)\right].
\end{equation}
\end{corollary}

\section{Dependence on the noise distribution}

Motivated by the spike examples of \cite{PaulinWhalley2026}, we examine whether higher moments affect the actual bias. Consider the Gaussian target $\pi=\mathcal{N}(0,1)$ and the stochastic gradient error $\eta_p$ defined simply by
\[
\mathbb{P}(\eta_p=0)=1-p,\qquad \mathbb{P}(\eta_p=\sigma/\sqrt p) =\mathbb{P}(\eta_p=-\sigma/\sqrt p)=p/2,
\]
where $\sigma>0$ and $0<p\leq 1$; then $\E\eta_p=0$, $\E\eta_p^2=\sigma^2$ and $\E\eta_p^4=\sigma^4/p$. Thus decreasing $p$ makes the noise increasingly concentrated in rare, large spikes, while preserving its variance.

For $0<h\leq1$, the stationary position admits the independent-sum representation
\[
X_{h,p}\overset{d}=G_h+S_{h,p},\qquad S_{h,p}=-h\sum_{j=0}^{\infty}(1-h)^j\eta_{p,j},
\]
where the $\eta_{p,j}$ are independent copies of $\eta_p$, independent of
\[
G_h\sim\mathcal{N}\left(0,\frac{2}{2-h}\right)=\pi_h.
\]
We have that $\E S_{h,p}^2=\frac{h\sigma^2}{2-h}$, and hence the stationary variance does not depend on $p$.

\begin{lemma}\label{lem:spike_example}
Let $\widetilde\pi_{h,p}$ be the invariant measure of \eqref{eq:SGLD} in one dimension with stepsize $0<h\leq 1$, a standard Gaussian target and stochastic gradient error governed by $\eta_p$, and $\pi_h$ be the invariant measure of \eqref{eq:ULA} with the same target; then we have that 
\[
\sup_{0<p\leq 1}W_2(\widetilde\pi_{h,p},\pi_h) =\sigma\sqrt{\frac{h}{2-h}}.
\] 
\end{lemma}

Lemma \ref{lem:spike_example} (whose proof we defer to the appendix) shows that the bias remains bounded at fixed $h$, but no first-order bound can hold uniformly over this fixed-variance family. This establishes the sharpness of the square-root stepsize order uniformly over noise laws with fixed variance. For fixed $p=10^{-8}$, Figure~\ref{fig:spiked_example} illustrates the crossover from approximately square-root to linear dependence on the stepsize as $h$ decreases.

\begin{figure}
\centering
\includegraphics[width=18pc]{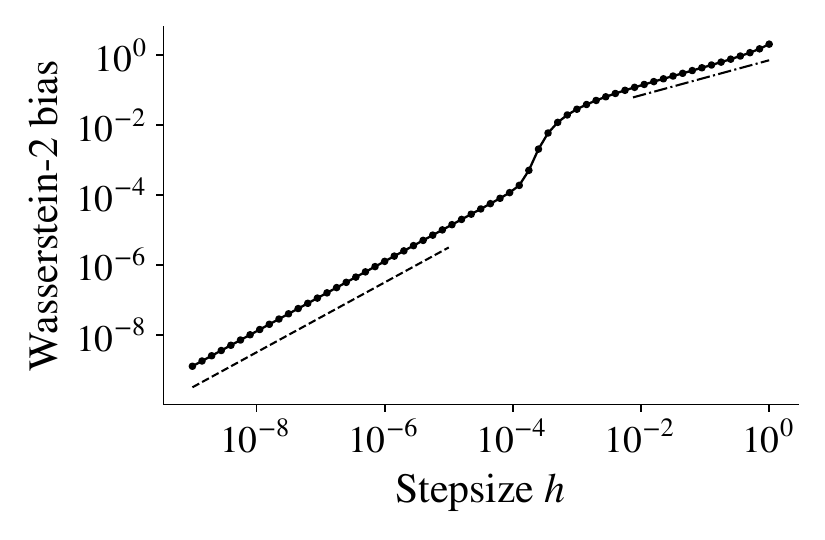}
\caption{Stationary bias $W_2(\widetilde\pi_{h,p},\pi)$ for $\pi=N(0,1)$, $p=10^{-8}$ and $\sigma=2$, computed by numerical inversion of the stationary characteristic function (see \cite{Davies1973}). The dashed and dot-dashed lines are reference curves proportional to $h$ and $h^{1/2}$, respectively.}
\label{fig:spiked_example}
\end{figure}

We can tensorise the same measure to provide a higher-dimensional example and illustrate adverse dimension dependence. We now consider \eqref{eq:SGLD} with target $\pi_d=\mathcal{N}(0,I_d)$, and independent stochastic gradient coordinate errors with the same spike law. Their covariance is $\sigma^2 I_d$, and tensorization yields
\[
\sup_{0<p\leq1} W_2(\widetilde\pi_{h,p}^{\otimes d},\pi_h^{\otimes d}) =\sigma\sqrt{\frac{hd}{2-h}}.
\]
At $h=d^{-1/2}$, this grows as $d^{1/4}$, whereas Gaussian gradient noise with the same covariance gives bias of order $h\sqrt d$, which remains bounded. Since the unadjusted Langevin bias also remains bounded at these stepsizes, the worst-case bias relative to $\pi_d$ has the same $d^{1/4}$ growth. These examples demonstrate the need for additional tail control in providing uniform first-order bounds.

\section*{Acknowledgement}
DP and PAW acknowledge the support of the Nanyang Technological University–University of Warwick Research Seed Fund.

\printbibliography

\appendix
\section{Proof of Lemma \ref{lem:spike_example}}

Fix $0<h\leq1$, and let $F_p$ and $F$ be the distribution functions of $\widetilde\pi_{h,p}$ and $\pi_h$, respectively.
Their quantile functions are denoted by $F_p^{-1}$ and $F^{-1}$ and we wish to compute 
\begin{align*}
W_2^2(\widetilde\pi_{h,p},\pi_h)
&=\int_0^1\left(F_p^{-1}(u)-F^{-1}(u)\right)^2du,
\end{align*}
and we have that
\[
\int_0^1 F_p^{-1}(u)^2du= \mathbb{E}X^{2}_{h,p} = \frac{2+h\sigma^2}{2-h},
\qquad
\int_0^1 F^{-1}(u)^2du= \mathbb{E}G^{2}_{h}=\frac{2}{2-h}.
\]
What remains is to control $\int_0^1F_p^{-1}(u)F^{-1}(u)du$, where 
\begin{align*}
    \left|\int_0^1F_p^{-1}(u)F^{-1}(u)du\right| \leq \left|\int_0^1\left(F_p^{-1}(u)-F^{-1}(u)\right)F^{-1}(u)du\right| + \int_0^1 F^{-1}(u)^{2}du.
\end{align*}
Considering the first term we have that for $M > 0$
\begin{align*}
\left|\int_{\left\{|F^{-1}|\leq M\right\}}\left(F_p^{-1}(u)-F^{-1}(u)\right)F^{-1}(u)du\right| &\leq M \int^{1}_{0}\left|F_p^{-1}(u)-F^{-1}(u)\right|du\\
&\leq M\mathbb{E}|S_{h,p}| \leq M\sigma\sqrt{p},
\end{align*}
due to optimality of the quantile coupling for the Wasserstein  distance of order one. Then we also have 
\begin{align*}
\left|\int_{\left\{|F^{-1}|> M\right\}}\left(F_p^{-1}(u)-F^{-1}(u)\right)F^{-1}(u)du\right| &\leq W_2(\widetilde\pi_{h,p},\pi_h)\left[\int_{\left\{|F^{-1}|> M\right\}}F^{-1}(u)^{2}du\right]^{1/2} \\
&\leq \sigma\sqrt{\frac{h}{2-h}}\left[\int_{\left\{|F^{-1}|> M\right\}}F^{-1}(u)^{2}du\right]^{1/2}\longrightarrow 0,
\end{align*}
as $M\to \infty$ by the dominated convergence theorem.

Now letting first $p\downarrow0$ and then $M\to\infty$ shows that
\[
\int_0^1F_p^{-1}(u)F^{-1}(u)du
\longrightarrow\frac{2}{2-h}.
\]
Hence in the same limit the one-dimensional representation of the Wasserstein distance yields
\begin{align*}
W_2^2(\widetilde\pi_{h,p},\pi_h)
&=\int_0^1\left(F_p^{-1}(u)-F^{-1}(u)\right)^2du\\
&=\frac{4+h\sigma^2}{2-h}
-2\int_0^1F_p^{-1}(u)F^{-1}(u)du \longrightarrow\frac{h\sigma^2}{2-h}.
\end{align*}
Coupling $G_h+S_{h,p}$ with $G_h$ also gives the matching upper bound, and therefore
\[
\sup_{0<p\leq 1}W^2_2(\widetilde\pi_{h,p},\pi_h) =\frac{h\sigma^2}{2-h},
\]
as required.

\end{document}